\documentclass[12pt]{amsart}
\usepackage{fancyhdr}
\usepackage{stmaryrd}
\usepackage{calrsfs}
\DeclareMathAlphabet{\pazocal}{OMS}{zplm}{m}{n}

\usepackage{amsmath}
\usepackage[nospace,noadjust]{cite}
\usepackage{amsfonts}
\usepackage{amssymb,enumerate}
\usepackage{amsthm}
\usepackage{cite}
\usepackage{comment}
\usepackage{color}
\usepackage[all]{xy}
\usepackage{hyperref}
\usepackage{lineno}
\usepackage{stmaryrd}
\usepackage{varwidth}

\usepackage{mathtools}
\usepackage{bm}
\usepackage{graphicx}
\usepackage{psfrag}
\usepackage{url}

\usepackage{fancyhdr}
\usepackage{tikz-cd}

\usepackage{dirtytalk}
\usepackage{float}
\usepackage{mathrsfs}

\newtheorem*{maintheorem*}{Main Theorem}
\newtheorem{theorem}{Theorem}[section]
\newtheorem{prop}[theorem]{Proposition}

\newtheorem{lemma}[theorem]{Lemma}

\theoremstyle{definition}
\newtheorem{definition}[theorem]{Definition}

\newtheorem{example}[theorem]{Example}
\numberwithin{equation}{section}

\newcommand{\cc}{\mathbb{C}}

\newcommand{\nn}{\mathbb{N}}
\newcommand{\qq}{\mathbb{Q}}
\newcommand{\rr}{\mathbb{R}}
\newcommand{\zz}{\mathbb{Z}}

\newcommand{\uu}{\mathcal{U}}

\newcommand{\red}{\text{red}}

\hypersetup{
	pdftitle={FD},
	colorlinks=true,
	linkcolor=blue,
	citecolor=cyan,
	urlcolor=wine
}



\begin{document}
	
	\mbox{}
	\title{Arithmetic of semisubtractive semidomains}

	\author{Hannah Fox}
	\address{}

	\author{Agastya Goel}
	
	\author{Sophia Liao}
	
	\date{\today}

	\begin{abstract}
		 A subset $S$ of an integral domain is called a semidomain if the pairs $(S,+)$ and $(S\setminus\{0\}, \cdot)$ are commutative and cancellative semigroups with identities. The multiplication of $S$ extends to the group of differences $\mathcal{G}(S)$, turning $\mathcal{G}(S)$ into an integral domain. In this paper, we study the arithmetic of semisubtractive semidomains (i.e., semidomains $S$ for which either $s \in S$ or $-s \in S$ for every $s \in \mathcal{G}(S)$). Specifically, we provide necessary and sufficient conditions for a semisubtractive semidomain to be atomic, to satisfy the ascending chain condition on principals ideals, to be a bounded factorization semidomain, and to be a finite factorization semidomain, which are subsequent relaxations of the property of having unique factorizations. In addition, we present a characterization of factorial and half-factorial semisubtractive semidomains. Throughout the article, we present examples to provide insight into the arithmetic aspects of semisubtractive semidomains.    
	\end{abstract}
	\medskip

	\maketitle


\bigskip
\section{Introduction}
\label{sec:intro}

Factorization theory studies the decomposition of elements into irreducible factors within various algebraic structures, including monoids~\cite{geroldingerreinhardt,gryinkiewicz}, semirings~\cite{BCG21, gottipolo}, integral domains~\cite{AAZ90,AAZ92}, and modules~\cite{baethwiegand, facchini}. A primary objective in this field is to gauge the extent to which an object deviates from having unique factorizations.

Anderson, Anderson, and Zafrullah~\cite{AAZ90} conducted the first systematic study of factorizations in the context of integral domains. In their paper, they introduced several relaxations of the unique factorization property and studied their preservation under algebraic constructions such as localization, polynomial extension, and the $D + M$ construction. Since then, several researchers have delved into the factorization properties of integral domains, as seen in~\cite{andersonmullins,coykendallzafrullah,halterkoch}. These subsequent studies reflect the sustained interest in understanding the factorization properties of integral domains.

A positive semiring is a subset of $\rr_{\geq 0}$ that contains the identities $0$ and $1$ and is closed under the standard operations of addition and multiplication. Positive semirings have generated much interest lately: Correa-Morris and Gotti~\cite{CGG20} investigated factorizations of positive algebraic valuations of $\nn_0[x]$, and Albizu-Campos et al.~\cite{ABP21} explored the factorization properties of positive semirings generated by some of the powers of a rational number. Moreover, Baeth et al.~\cite{BCG21} classified the additive and multiplicative factorizations of positive semirings. Finally, Baeth and Gotti~\cite{BG20} studied positive semirings in connection with factorizations of matrices.

A subset $S$ of an integral domain is called a semidomain if the pairs $(S,+)$ and $(S\setminus\{0\}, \cdot)$ are commutative and cancellative semigroups with identities. The multiplication of $S$ extends to the group of differences $\mathcal{G}(S)$, turning $\mathcal{G}(S)$ into an integral domain (\cite[Lemma~2.2]{gottipolo}). From now on, we refer to $\mathcal{G}(S)$ as the domain of differences of $S$. Observe that both integral domains and positive semirings belong to the class of semidomains, thus establishing a common ground to understand the similarities among these mathematical objects. In recent years, significant attention~\cite{CF19,CCMS09,chapmanpolo,gottipolo} has been devoted to investigating the arithmetic of semidomains with polynomial-like structures. 

In this paper, we study the factorization properties of semisubtractive semidomains (i.e., semidomains $S$ for which either $s \in S$ or $-s \in S$ for every $s \in \mathcal{G}(S)$). Previous literature has explored the algebraic and factorization properties of semisubtractive semidomains. In~\cite{dulin}, Dulin and Mosher showed that semisubtractive semidomains $S$ share many algebraic properties with their domains of differences $\mathcal{G}(S)$, while in~\cite{doverstone}, Dover and Stone introduced a generalization of the notion of semisubtractivity and, for semirings satisfying this condition, provided analogues of Artin-Wedderburn and Goldie structure theorems. Moreover, Stone~\cite{stone} and Alarcón and Anderson~\cite{alarconanderson} studied the ideals of certain classes of semisubtractive semirings.

Following \cite{pC2017}, a semidomain is called Furstenberg if every nonzero nonunit element is divisible by an irreducible. Furstenberg semidomains have generated some interest lately~\cite{lebowitz, gottipolo2, gottizafrullah}. In Section~\ref{sec: Furstenberg}, we examine semisubtractive semidomains that satisfy the Furstenberg property. Specifically, we show that a semisubtractive semidomain $S$ is Furstenberg if and only if $\mathcal{G}(S)$ is Furstenberg.

In Section~\ref{sec: atomicity}, we explore the conditions under which a semisubtractive semidomain is a finite factorization semidomain (resp., a bounded factorization semidomain). Let $S$ be a semidomain. If every nonzero element of $S$ has finitely many divisors up to associates, then we call $S$ a finite factorization semidomain, while $S$ is a bounded factorization semidomain provided that there exists a function $\ell\colon S\setminus\{0\} \to \nn_0$ satisfying that $\ell(s) = 0$ if and only if $s$ is a unit of $S$ and $\ell(ss') \geq \ell(s)+\ell(s')$ for elements $s,s' \in S$. A finite factorization semidomain is a bounded factorization semidomain.

 The notion of half-factoriality was introduced in~\cite{carlitz}, where it was proved that the principal order $\pazocal{O}_F$ of an algebraic number field $F$ is half-factorial (i.e., every nonzero element of $\pazocal{O}_F$ has factorizations of equal length) if and only if $F$ has class number at most two. Half-factorial domains have been extensively investigated~\cite{coykendall, aZ76,aZ80}. We conclude in Section~\ref{sec: factoriality}, where we completely describe when a semisubtractive semidomain is unique factorization. Additionally, we present a characterization of half-factorial semisubtractive semidomains.

\bigskip
\section{Preliminaries}
\label{sec:background}

This section introduces the notation and terminology that will be used later. Readers who seek a comprehensive background on factorization theory and semiring theory can refer to \cite{GH06} and \cite{JG1999}, respectively. 

\subsection{General Notation}

The standard notations used are $\zz$, $\qq$, $\rr$, and $\mathbb{P}$, which represent the sets of integers, rational numbers, real numbers, and prime numbers, respectively. Additionally, we use $\nn$ to denote the set of positive integers and $\nn_0$ to denote the set of non-negative integers. Given $r \in \rr$ and a subset $S \subseteq \rr$, we use $S_{>r}$ to denote the set of elements in $S$ that are greater than $r$. Similarly, $S_{\geq r}$ is defined as the set of elements greater than or equal to $r$. For $m, n \in \nn_0$, we use the notation $\llbracket m,n \rrbracket$ to represent the discrete interval from $m$ to $n$. That is, $\llbracket m,n \rrbracket$ is defined as $\{k \in \zz \mid m \leq k \leq n\}$. 


\subsection{Monoids and Factorizations}

In this paper, we define a \emph{monoid}\footnote{Note that the standard definition of a monoid does not assume the cancellative and commutative conditions.} to be a cancellative and commutative semigroup with identity. For convenience, we use multiplicative notation for monoids, unless otherwise specified. Let $M$ be a monoid with identity $1$. A subsemigroup of $M$ containing $1$ is called a \emph{submonoid}. We write $M^{\times}$ for the group of units of $M$. We also use $M_{\red}$ to denote the quotient $M/M^{\times}$, which is also a monoid. We say that $M$ is \emph{reduced} if $M^{\times}$ is the trivial group, in which case we identify $M_{\red}$ with $M$. The \emph{Grothendieck group} of~$M$ is an abelian group $\mathcal{G}(M)$ such that there exists a monoid homomorphism $\iota \colon M \to \mathcal{G}(M)$ satisfying the following universal property: for any monoid homomorphism $f \colon M \to G$, where $G$ is an abelian group, there exists a unique group homomorphism $g \colon \mathcal{G}(M) \to G$ such that $f = g \circ \iota$. The Grothendieck group of a monoid is unique up to isomorphism. For a subset $S$ of $M$, we let $\langle S \rangle$ denote the smallest submonoid of $M$ containing~$S$, and $S$ is a \emph{generating set} of~$M$ if $M = \langle S \rangle$. We say that $M$ is \emph{torsion-free} if $a^n=b^n$ implies $a=b$ for all $a,b\in M$ and $n\in\nn$.

For elements $b,c \in M$, it is said that $b$ \emph{divides} $c$ \emph{in} $M$, denoted by $b \mid_M c$, if there exists $c' \in M$ such that $c = c'b$. Two elements $b,c \in M$ are \emph{associates} in $M$, which we denote by $b \simeq_M c$, provided that $b \mid_M c$ and $c \mid_M b$. The notation established in this paragraph omits the subscript whenever the monoid can be inferred from the context. 

We define an \emph{atom} in a monoid $M$ as an element $a\in M\setminus M^{\times}$ such that for any $b,c\in M$, the equality $a=bc$ implies that either $b\in M^{\times}$ or $c\in M^{\times}$. Let $\mathcal{A}(M)$ be the set of all atoms of $M$. Following~\cite{pC2017}, we say that the monoid $M$ is \emph{Furstenberg} if every nonunit element is divisible by an atom, while $M$ is said to be \emph{atomic} if every element of $M\setminus M^{\times}$ can be expressed as a (finite) product of atoms. Clearly, an atomic monoid is Furstenberg. Moreover, one can easily verify that $M$ is atomic if and only if $M_{\red}$ is atomic.

An \emph{ideal} of a monoid $M$ is a subset $I\subseteq M$ satisfying $IM\subseteq I$ or, equivalently, $IM=I$. If $I$ is an ideal of $M$ such that $I=bM$ for some $b\in M$, then $I$ is called a \emph{principal ideal}. The monoid $M$ satisfies the \emph{ascending chain condition on principal ideals} (\emph{ACCP}) if every increasing sequence of principal ideals of $M$ (under inclusion) becomes constant from one point on. The ACCP is equivalent to the following condition:

\smallskip
If $(b_i)_{i \in \nn}$ is a sequence of elements of $M$ such that $b_{i + 1} \mid b_i$ for every $i \in \nn$, then there exists some $k \in \nn$ such that $b_n \simeq b_k$ for all $n \geq k$.
\smallskip

\noindent It is well known that monoids satisfying the ACCP are atomic.

Assuming that $M$ is an atomic monoid, let $\mathsf{Z}(M)$ denote the free commutative monoid generated by the set of atoms $\mathcal{A}(M_{\red})$. The elements of $\mathsf{Z}(M)$ are called factorizations, and the length $\ell$ of a factorization $z=a_1\cdots a_{\ell}$, where $a_1,\dots,a_{\ell}\in\mathcal{A}(M_{\red})$, is denoted by $|z|$. There is a unique monoid homomorphism $\pi:\mathsf{Z}(M)\rightarrow M_{\red}$ that maps every atom to itself. For each $b\in M$, we define two sets that are essential for the study of factorization theory:
\[
	\mathsf{Z}_M(b) = \pi^{-1}(bM^{\times})\subseteq\mathsf{Z}(M) \hspace{.3 cm}\text{ and } \hspace{.3 cm} \mathsf{L}_M(b) = \{|z|\mid z\in\mathsf{Z}_M(b)\}\subseteq\nn_0.
\]
As usual, we drop the subscript $M$ if there is no risk of ambiguity. We use the following terminology: $M$ is a \emph{finite factorization monoid} (\emph{FFM}) if $\mathsf{Z}(b)$ is finite for all $b\in M$; $M$ is a \emph{bounded factorization monoid} (\emph{BFM}) if $\mathsf{L}(b)$ is finite for all $b\in M$; $M$ is a \emph{half-factorial monoid} (\emph{HFM}) if $|\mathsf{L}(b)|=1$ for all $b\in M$; and $M$ is a \emph{unique factorization monoid} (\emph{UFM}) if $|\mathsf{Z}(b)|=1$ for all $b\in M$. A unique factorization monoid is also called \emph{factorial}. It follows from the definitions that every FFM is a BFM, every BFM satisfies the ACCP, every UFM is an HFM, and every HFM is a BFM.

\subsection{Semirings and Semidomains}

A \emph{commutative semiring} $S$ is defined as a non-empty set with two binary operations: addition denoted by ``+" and multiplication denoted by ``$\cdot$". The properties of a commutative semiring are:

\begin{enumerate}
	\item $(S,+)$ forms a monoid with an identity element 0;
	\item $(S,\cdot)$ forms a commutative semigroup with an identity element 1, where 1 $\neq$ 0;
	\item $s_1 \cdot (s_2+s_3)= s_1 \cdot s_2 + s_1 \cdot s_3$ for all $s_1, s_2, s_3 \in S$.
\end{enumerate}

If $S$ is a commutative semiring, then the distributive law and the cancellative property of addition in $S$ imply that $0 \cdot s = 0$ for all $s \in S$. In this paper, $ss'$ is used interchangeably with $s\cdot s'$ unless it causes confusion. In this context, it is worth noting that the conventional definition of a semiring does not necessarily require the semigroup $(S,+)$ to be cancellative. However, in our particular case, the semirings of interest do possess cancellative additive structures. Similarly, the general definition of a semiring $S$ does not impose the condition of commutativity on the semigroup $(S,\cdot)$. Nonetheless, for the scope of this paper, we are primarily concerned with commutative semirings. Hence, we will employ the term \emph{semiring} to specifically refer to a commutative semiring, assuming commutativity for both operations. 

A subset $S'$ of $S$ is called a \emph{subsemiring} if $(S',+)$ is a submonoid of $(S,+)$ that is closed under multiplication and contains 1. Every subsemiring of $S$ is a semiring. A semiring $S$ is called a \emph{semidomain} if it is a subsemiring of an integral domain. If $S$ is a semidomain, then $(S \setminus \{0\},\cdot)$ forms a monoid, denoted by $S^*$ and referred to as the \emph{multiplicative monoid of $S$}. We use the term \emph{invertible elements} of $S$ to refer to the units of the monoid $(S,+)$ and simply \emph{units} of $S$ to refer to the units of the multiplicative monoid $S^*$, to avoid confusion. The group of units of $S$ is denoted by $S^\times$, while the additive group of invertible elements of $S$ is denoted by $\uu(S)$. We use $\mathcal{A}(S)$ instead of $\mathcal{A}(S^*)$ to refer to the set of atoms of the multiplicative monoid $S^*$, and we write $s \mid_S s'$ to indicate that $s$ divides $s'$ in $S^*$.

\begin{lemma} \label{lem:characterization of integral semirings} ~\cite[Lemma 2.2]{gottipolo}
	For a semiring $S$, the following conditions are equivalent.
	\begin{enumerate}
		\item[(a)] $S$ is a semidomain.
		
		\item[(b)] The multiplication of $S$ extends to the Grothendieck group $\mathcal{G}(S)$ of $(S,+)$ turning $\mathcal{G}(S)$ into an integral domain.
	\end{enumerate}
\end{lemma}

A semidomain $S$ is called Furstenberg provided that $S^*$ is a Furstenberg monoid. If the multiplicative monoid $S^*$ of $S$ satisfies atomicity or the ACCP, we call $S$ atomic or ACCP, respectively. We refer to $S$ as a BFS, an FFS, an HFS, or a UFS if $S^*$ is a BFM, an FFM, an HFM, or a UFM, respectively. It is worth noting that if $S$ is an integral domain, then we recover the established notions of UFD, BFD, FFD, and HFD. However, even though a semidomain $S$ can be embedded in an integral domain $R$, $S$ may not inherit any atomic property from $R$ since the integral domain $\qq[x]$ is a UFD, but its subring $\zz + x\qq[x]$ is not even atomic.

\begin{definition}
	Let $S$ be a semidomain. We say that $S$ is \emph{semisubtractive} if, for every $s \in \mathcal{G}(S)$, we have that either $s \in S$ or $-s \in S$.
\end{definition}

\subsection{Localization of Semidomains}

Let $S$ be a semidomain, and let $D$ be a \emph{multiplicative subset} of $S$ (i.e., a submonoid of $S^*$). Since $D$ is also a multiplicative subset of $\mathcal{G}(S)$, we can consider the localization of $\mathcal{G}(S)$ at $D$, which we denote by $D^{-1}\mathcal{G}(S)$. Set $R = (S \times D)/\sim$, where $\sim$ is an equivalence relation on $S \times D$ defined by $(s,d) \sim (s',d')$ if and only if $sd' = ds'$. We let $\frac{s}{d}$ denote the equivalence class of $(s,d)$. Define the following operations in $R$:
\[
\frac{s}{d} \cdot \frac{s'}{d'} = \frac{ss'}{dd'} \hspace{.3 cm} \text{ and } \hspace{.3 cm} \frac{s}{d} + \frac{s'}{d'} = \frac{sd' + ds'}{dd'}.
\]
It is routine to verify that these operations are well defined and that $(R,+,\cdot)$ is a semiring\footnote{The localization of semirings is presented in greater generality in \cite[Chapter~11]{JG1999}.}. The semiring $R$ is called \emph{the localization of $S$ at $D$} and is denoted by $D^{-1}S$. Let $\varphi\colon R \to D^{-1}\mathcal{G}(S)$ be a function given by $\varphi(\frac sd) = \overline{\frac sd}$, where $\overline{\frac sd}$ represents the equivalence class of $(s,d)$ as an element of $D^{-1}\mathcal{G}(S)$. Note that $\varphi$ is a well-defined semiring homomorphism. Since $\varphi$ is injective, $R$ is a semidomain. 

\begin{lemma}\label{lemma: localization of semisubtractive semidomains}
	The localization of a semisubtractive semidomain is semisubtractive.
\end{lemma}

\begin{proof}
	Suppose that $S$ is a semisubtractive semidomain, and let $D$ be a multiplicative subset of $S$. Let $\frac{s}{d} \in D^{-1}\mathcal{G}(S)$ such that $\frac{s}{d} \notin D^{-1}S$ (which implies that $s \notin S$). Since $S$ is semisubtractive, we have that $-s \in S$. This implies that $\frac{-s}{d} = -\frac{s}{d} \in D^{-1}S$, which concludes our argument.
\end{proof}

\section{Furstenbergness} \label{sec: Furstenberg}

In this section, we study semisubtractive semidomains that satisfy the Furstenberg property. Specifically, we show that a semisubtractive semidomain is Furstenberg if and only if $\mathcal{G}(S)$ is Furstenberg. Before discussing the main result of this section, we establish some lemmas concerning the units and irreducibles of a semisubtractive semidomain.

\begin{lemma} \label{lemma: units}
	Let $S$ be a semisubtractive semidomain. Then $S^{\times} = S \cap \mathcal{G}(S)^{\times}$\!.
\end{lemma}

\begin{proof}
	The inclusion $S^{\times} \subseteq S \cap \mathcal{G}(S)^{\times}$ clearly holds. Now let $u \in S \cap \mathcal{G}(S)^{\times}$\!. There exists $s \in \mathcal{G}(S)$ such that $us = 1$. Without loss of generality, we may assume that $s \not\in S$, which implies that $-s \in S$ as $S$ is semisubtractive. Since $u(-s) = -1$ is an element of $S$, we have $S = \mathcal{G}(S)$, which concludes our argument.
\end{proof}

Given a semisubtractive semidomain $S$, an atom of $S$ is not necessarily an atom of $\mathcal{G}(S)$ as the following example illustrates.

\begin{example} \label{ex: atom in S not in gp(S)}
	Consider the semidomain 
	\[
	S = \left\{c_nx^n + \cdots + c_1x + c_0 \in \zz[x] \mid \text{ either } c_0 > 0 \text{ or } c_1 \geq c_0 = 0\right\}\!.
	\]
	It is worth noting that $S$ is a semisubtractive semidomain such that $\mathcal{G}(S) \cong \zz[x]$. Observe that $-x^2 \in \mathcal{A}(S) \setminus \mathcal{A}(\zz[x])$. 
\end{example}

Example~\ref{ex: atom in S not in gp(S)} shows that, in contrast to units, atoms present a greater challenge when it comes to their description.

\begin{lemma} \label{lemma: atoms}
	Let $S$ be a semisubtractive semidomain. The following statements hold.
	\begin{enumerate}
		\item $S \cap \mathcal{A}(\mathcal{G}(S)) \subseteq \mathcal{A}(S)$.
		\item If $a \in \mathcal{A}(S) \setminus \mathcal{A}(\mathcal{G}(S))$, then $-a \in S^* \setminus (S^{\times}\cup \mathcal{A}(S))$.
	\end{enumerate}
\end{lemma}

\begin{proof}
	The first part follows readily from Lemma~\ref{lemma: units}. As for $(2)$, let $a \in \mathcal{A}(S) \setminus \mathcal{A}(\mathcal{G}(S))$, and write $a = s_1 s_2$ for some nonunits $s_1, s_2 \in \mathcal{G}(S)$. Note that there is no loss in assuming that $s_1 \not\in S$ which, in turn, implies that $-s_1 \in S$. Consequently, we have that $-s_2 \not\in S$; otherwise, we would have $a = (-s_1)(-s_2)$ but neither $-s_1$ nor $-s_2$ is a unit of $S$ by Lemma~\ref{lemma: units}. Since $s_2 \in S$, we obtain that $-a = (-s_1)s_2$ is a nonzero nonunit element of $S$ (that is obviously not irreducible).
\end{proof}

Now we are in a position to prove the main result of this section.

\begin{theorem} \label{theorem: Furstenberg}
	Let $S$ be a semisubtractive semidomain. Then $S$ is Furstenberg if and only if $\mathcal{G}(S)$ is Furstenberg.
\end{theorem}

\begin{proof}
	Suppose that $S$ is Furstenberg. Let $s$ be a nonzero nonunit element of $\mathcal{G}(S)$, and assume that $s \in S$. We claim that there exists $a \in \mathcal{A}(\mathcal{G}(S))$ such that $a \mid_{\mathcal{G}(S)} s$. Since $S$ is Furstenberg, there exists $a \in \mathcal{A}(S)$ such that $a \mid_S s$. Without loss of generality, assume that $a \notin \mathcal{A}(\mathcal{G}(S))$. By Lemma~\ref{lemma: atoms}, we have that $-a \in S^* \setminus (S^{\times}\cup \mathcal{A}(S))$. Again, since $S$ is Furstenberg, there exists $a_1 \in \mathcal{A}(S)$ such that $a_1 \mid_S -a$. Observe that if $-a_1 \in S$, then $-a_1 \mid_S a$ and $-a_1 \not\simeq_S a$, but this contradicts that $a \in \mathcal{A}(S)$. We can then conclude that $-a_1 \notin S$. By Lemma~\ref{lemma: atoms}, we have that $a_1 \in \mathcal{A}(\mathcal{G}(S))$. Note that $a_1 \mid_{\mathcal{G}(S)} s$, which proves our claim. On the other hand, if $s \notin S$, then $-s \in S$ and there exists $a \in \mathcal{A}(\mathcal{G}(S))$ such that $a \mid_{\mathcal{G}(S)} -s$. This, in turn, implies that $a \mid_{\mathcal{G}(S)} s$. Therefore, $\mathcal{G}(S)$ is Furstenberg.
	
	To tackle the reverse implication, assume that $\mathcal{G}(S)$ is Furstenberg and pick a nonzero nonunit element $s \in S$. There is no loss in assuming that $s$ is not an atom of $S$. Then there exist nonzero nonunit elements $s_1, s_2 \in S$ such that $s = s_1s_2$. Since $\mathcal{G}(S)$ is Furstenberg, there exist $a_1, a_2 \in \mathcal{A}(\mathcal{G}(S))$ such that $a_1 \mid_{\mathcal{G}(S)} s_1$ and $a_2 \mid_{\mathcal{G}(S)} s_2$. In other words, there exist $s'_1, s'_2 \in \mathcal{G}(S)$ such that $s_1 = a_1s'_1$ and $s_2 = a_2s'_2$. If $a_1, s'_1 \in S$ (resp., $a_1, s'_1 \notin S$) or $a_2, s'_2 \in S$ (resp., $a_2, s'_2 \notin S$), then our argument concludes by virtue of Lemma~\ref{lemma: atoms}. Consequently, we may assume that $-a_1s'_1$ and $-a_2s'_2$ are elements of $S$. If $a_1 \in S$, then $-s'_1 \in S$, which implies that $a_1 \mid_S s$, where $a_1 \in S \cap \mathcal{A}(\mathcal{G}(S)) \subseteq \mathcal{A}(S)$ (Lemma~\ref{lemma: atoms}). On the other hand, if $a_1 \not\in S$, then $-a_1 \in S$ and $s'_1 \in S$, which implies that $-a_1 \mid_S s$, where $-a_1 \in S \cap \mathcal{A}(\mathcal{G}(S)) \subseteq \mathcal{A}(S)$. Either way, we can conclude that $S$ is Furstenberg.
\end{proof}

Using Theorem~\ref{theorem: Furstenberg}, we can now provide new instances of Furstenberg semidomains that are not integral domains. Consider the following example.

\begin{example}
	Consider the semidomain $S = \nn_0 + x\qq[x]$, which is not atomic. Indeed, note that $n \mid_S x$ for every $n \in \nn$; consequently, the element $x \in S$ cannot be factored as a finite product of atoms. On the other hand, we have that $\mathcal{G}(S) = \zz + x\qq[x]$, which is Furstenberg by \cite[Lemma~16]{lebowitz}. Since $S$ is a semisubtractive semidomain, $S$ is Furstenberg by Theorem~\ref{theorem: Furstenberg}. 
\end{example}

\section{Atomicity} \label{sec: atomicity}

We now explore the conditions under which a semisubtractive semidomain can be classified as atomic. Furthermore, we provide necessary and sufficient conditions for a semisubtractive semidomain to satisfy the ACCP, be a BFS, and be an FFS.

\begin{lemma} \label{lemma: atomicity}
Let $S$ be a semisubtractive semidomain such that $\mathcal{G}(S)$ is atomic. The following statements hold.
\begin{enumerate}
    \item Each nonunit $s\in S\setminus\mathcal{U}(S)$ is a finite product of atoms of $S$.
    \item For each nonzero nonunit $s\in\mathcal{U}(S)$, either $s$ is a finite product of atoms of $S$ or $-s$ is a finite product of atoms of $S$.
\end{enumerate}
\end{lemma}

\begin{proof}
Let $s$ be a nonzero nonunit element of $S$, and let $a_1\cdots a_n$ be a factorization of $s$ in $\mathcal{G}(S)$. Choose $e_1,\ldots,e_n\in\{0,1\}$ such that $a_i'=(-1)^{e_i}a_i$ is an element of $S$ for all $i\in[\![1,n]\!]$. Observe that $a_1'\cdots a_n'=(-1)^{e_1+\cdots+e_n}a$ is an element of $S$. If $s\notin\mathcal{U}(S)$, then $e_1+\cdots+e_n$ is even. Consequently, it follows from Lemma~\ref{lemma: atoms} that $a_1'\cdots a_n'$ is a factorization of $s$ in $S$. Similarly, if $s\in\mathcal{U}(S)$, then $a_1'\cdots a_n'$ is either a factorization of $s$ or $-s$ in $S$.
\end{proof}

We are now in a position to describe semisubtractive semidomains that are atomic.

\begin{theorem} \label{theorem: atomicity iff}
Let $S$ be a semisubtractive semidomain. The following statements are equivalent.
\begin{enumerate}
    \item $S$ is atomic.
    \item $\mathcal{G}(S)$ is atomic and for each $s\in\mathcal{U}(S)$, if $s$ is a finite product of atoms of $S$, then $-s$ is a finite product of atoms of $S$.
\end{enumerate}
\end{theorem}

\begin{proof}
    Suppose that $S$ is atomic, and let $a \in \mathcal{A}(S)$. If $a \not\in\mathcal{A}(\mathcal{G}(S))$, then $-a \in S^* \setminus (S^{\times}\cup \mathcal{A}(S))$ by Lemma~\ref{lemma: atoms}. Since $S$ is atomic, we have $-a = a_1 \cdots a_n$, where $n \in \nn_{>1}$ and $a_i \in \mathcal{A}(S)$ for every $i \in \llbracket 1,n \rrbracket$. By way of contradiction, assume that there exists $j \in \llbracket 1,n \rrbracket$ such that $a_j \not\in\mathcal{A}(\mathcal{G}(S))$. By Lemma~\ref{lemma: atoms}, we have $-a_j = a_1' \cdots a_k'$ with $k \in \nn_{>1}$ and $a_t' \in \mathcal{A}(S)$ for all $t \in \llbracket 1,k \rrbracket$. Thus, 
	\[
	a = a_1 \cdots a_{j - 1}(a_1' \cdots a_k')a_{j + 1} \cdots a_n,
	\]
	but this contradicts that $a \in \mathcal{A}(S)$ since $k + n > 2$. Hence if $a \in \mathcal{A}(S)$, then either $a \in \mathcal{A}(\mathcal{G}(S))$ or $-a \in \langle S \cap \mathcal{A}(\mathcal{G}(S)) \rangle$. Now let $s$ be a nonzero nonunit element of $\mathcal{G}(S)$. If $s \in S$, then $s = a_1 \cdots a_n$ with $a_i \in \mathcal{A}(S)$ for each $i \in \llbracket 1,n \rrbracket$. By our previous observation, either $s \in \langle \mathcal{A}(\mathcal{G}(S)) \rangle$ or $-s \in \langle \mathcal{A}(\mathcal{G}(S)) \rangle$. In any case, we have that $s \in \langle \mathcal{A}(\mathcal{G}(S)) \rangle$. On the other hand, if $s \not\in S$, then $-s \in S$, and we proceed similarly. Therefore, $\mathcal{G}(S)$ is atomic. Note that the latter condition in (2) is clearly satisfied, as $S$ is atomic.

    The reverse implication is an immediate consequence of Lemma~\ref{lemma: atomicity}.
\end{proof}

The rest of this section is devoted to the study of the ACCP and the bounded and finite factorization properties in the context of semisubtractive semidomains. We start by stating a well-known characterization of BFMs.

\begin{definition}
	Given a monoid $M$, a function $\ell\colon M \to \nn_0$ is a \emph{length function} of $M$ if it satisfies the following two properties:
	\begin{enumerate}
		\item[(i)] for every $u\in M$, $\ell(u) = 0$ if and only if $u \in M^{\times}$;
		
		\item[(ii)] $\ell(bc) \geq \ell(b) + \ell(c)$ for every $b,c \in M$.
	\end{enumerate}
\end{definition} 

As we mentioned before, the following result is well known.

\begin{prop}[\cite{fHK92}, Theorem~1] \label{prop: length function if and only if BFM}
	A monoid $M$ is a BFM if and only if there is a length function $\ell\colon M \to \nn_0$.
\end{prop}

We can now provide necessary and sufficient conditions for a semisubtractive semidomain to satisfy the ACCP, be a BFS, and be an FFS.

\begin{theorem} \label{theorem: ACCP, BFP, and FFP}
	Let $S$ be a semisubtractive semidomain. The following statements hold.
	\begin{enumerate}
		\item $S$ satisfies the ACCP if and only if $\mathcal{G}(S)$ satisfies the ACCP.
		\item $S$ is a BFS if and only if $\mathcal{G}(S)$ is a BFD.
		\item $S$ is an FFS if and only if $\mathcal{G}(S)$ is an FFD.
	\end{enumerate}
\end{theorem}

\begin{proof}
	The reverse implication of part $(1)$ follows from Lemma~\ref{lemma: units}. Suppose now that $S$ satisfies the ACCP. Let $s_1, s_2, \ldots$ be a sequence of elements of $\mathcal{G}(S)^*$ such that, for each $i \in \nn$, there exists $s'_{i + 1} \in \mathcal{G}(S)$ satisfying that $s_i = s_{i + 1} s'_{i + 1}$. There is no loss in assuming that $s_i \in S$ for every $i \in \nn$ given that $S$ is semisubtractive. Since $S$ satisfies the ACCP and the equality $S^{\times} = S \cap \mathcal{G}(S)^{\times}$ holds, we may further assume that for every $k \in \nn$ there exists $m \in \nn_{>k}$ such that $s'_m \not\in S$. Next we will define subsequences of $(s_n)_{n \in \nn}$ and $(s'_n)_{n \in \nn_{>1}}$ recursively. Keep in mind that the latter subsequence will be obtained as a byproduct of the former. Let $s_{k_1} = s_1$. Suppose that, for some $n \in \nn$, we already defined $s_{k_n} = s_m$ for some $m \in \nn$. Let $u,v \in \nn_{>m}$ be indices satisfying that $u < v$, the elements $s'_u$ and $s'_v$ are not in $S$, and $s'_t \in S$ for every $t \in \llbracket m + 1, v \rrbracket \setminus \{u,v\}$. Thus,
	\[
		s_{k_n} = s_m = s_vs'_v \cdots s'_u \cdots s'_{m + 1} = s_v(-s'_v)\cdots (-s'_u) \cdots s'_{m + 1}.
	\] 
	Set $s_{k_{n + 1}} \coloneqq s_v$ and $s'_{k_{n + 1}}\coloneqq (-s'_v)\cdots (-s'_u) \cdots s'_{m + 1}$. Clearly, $(s_{k_n})_{n \in \nn}$ is a sequence of elements of $S$ satisfying that $s_{k_{n + 1}} \mid_S s_{k_n}$ for every $n \in \nn$. Since $S$ satisfies the ACCP, there exists $m \in \nn$ such that $s_{k_n} \simeq_S s_{k_m}$ for every $n \in \nn_{\geq m}$. This implies that $(s_n)_{n \in \nn}$ (as a sequence of elements of $\mathcal{G}(S)$) eventually stabilizes by Lemma~\ref{lemma: units}. Hence $\mathcal{G}(S)$ satisfies the ACCP.
	
	Note that the reverse implication of part $(2)$ follows from Lemma~\ref{lemma: units} and \cite[Corollary~1.3.3]{GH06}. To tackle the forward implication, assume that $S^{*}$ is a BFM. By Proposition~\ref{prop: length function if and only if BFM}, there exists a length function $\ell \colon S^{*} \to \nn_0$. Consider the function $\ell^2\colon \mathcal{G}(S)^* \to \nn_0$ given by $\ell^2(s) = \ell(s^2)$, which is well defined because $S$ is semisubtractive. For $u \in \mathcal{G}(S)^{\times}$, we have that $u^2 \in S \cap \mathcal{G}(S)^{\times}$. By Lemma~\ref{lemma: units}, we have that $\ell^2(u) = \ell(u^2) = 0$. Conversely, let $s \in \mathcal{G}(S)$ such that $\ell^2(s) = 0$. Thus, for some $e \in \{0,1\}$, we have
	\[
		0 = \ell^2(s) = \ell(s^2) \geq 2\ell((-1)^es),
	\]
	where $(-1)^es \in S$. This implies that $\ell((-1)^es) = 0$. In other words, $(-1)^es$ is a unit of $S$. We can then conclude that $s \in \mathcal{G}(S)^{\times}$. Moreover, for $s_1, s_2 \in \mathcal{G}(S)^*$, we have
	\[
		\ell^2(s_1s_2) = \ell\left(s_1^2s_2^2\right) \geq \ell\left(s_1^2\right) + \ell\left(s_2^2\right) = \ell^2(s_1) + \ell^2(s_2).
	\]
	As a consequence, we have that $\ell^2$ is a length function. Then the domain $\mathcal{G}(S)$ is a BFD.
	
	Without loss of generality, we can assume that $S$ is not an integral domain; otherwise, $\mathcal{G}(S) = S$. Take $u \in \mathcal{G}(S)^{\times} \setminus S^{\times}$\!, and let $s \in \mathcal{G}(S)^{\times}$ such that $us = 1$. Note that $-u \in S$ by virtue of Lemma~\ref{lemma: units}, which implies that $s \not\in S$. This, in turn, implies that $-s \in S$, so $-u \in S^{\times}$. Hence $(\mathcal{G}(S)^{\times} : S^{\times}) = 2 < \infty$, and then the reverse implication of $(3)$ follows from \cite[Theorem~1.5.6(2)]{GH06}. 
	
	Conversely, suppose that $\mathcal{G}(S)$ is not an FFD. By \cite[Proposition~1.5.5(1)]{GH06}, there exists $s \in \mathcal{G}(S)^*$ such that $s$ has infinitely many non-associated divisors in $\mathcal{G}(S)^*$. Let $(s_n)_{n \in \nn}$ be a sequence of non-associated divisors of $s$ in $\mathcal{G}(S)^*$, i.e., for each $n \in \nn$ there exists $s'_n \in \mathcal{G}(S)^*$ such that $s = s_ns'_n$ and $s_n \not\simeq_{\mathcal{G}(S)} s_m$ for $n \neq m$. For every $n \in \nn$, let $e_n, k_n \in \{0,1\}$ such that $(-1)^{e_n}s_n$ and $(-1)^{k_n}s'_n$ are both elements of $S$. By possibly taking a subsequence of $(s_n)_{n \in \nn}$, there is no loss in assuming that $k_i = k_j$ and $e_i = e_j$ for $i,j \in \nn$. Since $s_i \not\simeq_{\mathcal{G}(S)} s_j$ for $i \neq j$, we have that $(-1)^{e_i}s_i \not\simeq_S (-1)^{e_j}s_j$ for $i \neq j$ by Lemma~\ref{lemma: units}. Hence the element $s' = (-1)^{e_i}s_i(-1)^{k_i}s'_i \in S$ has infinitely many non-associated divisors in $S$, namely $(-1)^{e_i}s_i$ for each $i \in \nn$. Therefore, $S$ is not an FFS by virtue of \cite[Proposition~1.5.5(1)]{GH06}.
\end{proof}

We conclude this section by offering some examples that distinguish the properties considered in Theorem~\ref{theorem: ACCP, BFP, and FFP}.
\begin{example}
	Let $D = \rr + x \cc[x]$. It is known that $D$ is a BFD (in fact, it is half-factorial) that is not an FFD (see, for instance, \cite[Example~4.10]{AG22}). Let $S = \rr_{\geq 0} + x\cc[x]$, which is clearly a semisubtractive semidomain satisfying that $\mathcal{G}(S) \cong D$. Therefore, the semidomain $S$ is a BFS that is not an FFS by Theorem~\ref{theorem: ACCP, BFP, and FFP}.
\end{example}

\begin{example}
	Given a semidomain $S$ and a torsion-free monoid $M$ written additively, consider the set $S[M]$ consisting of all maps $f\colon M \to S$ satisfying that the set $\{m \in M \mid f(m) \neq 0\}$ is finite. Addition and multiplication in $S[M]$ are defined as for polynomials and, under these operations, $S[M]$ is a semidomain (see \cite{chapmanpolo} and \cite{rG84} for more details about this construction). Let $M = \langle 1/p \mid p \in \mathbb{P} \rangle$, and consider $D = \qq[M]$. The integral domain $D$ satisfies the ACCP but it is not a BFD (\cite[Example~4.8]{AG22}). By virtue of Theorem~\ref{theorem: ACCP, BFP, and FFP}, the semisubtractive semidomain $S = \{f \in \qq[M] \mid f(0) \geq 0\}$ satisfies the ACCP but is not a BFS.
\end{example}

\section{Factoriality}
\label{sec: factoriality}

In Example~\ref{ex: atom in S not in gp(S)}, we consider the semisubtractive semidomain
\[
	S = \left\{c_nx^n + \cdots + c_1x + c_0 \in \zz[x] \mid \text{ either } c_0 > 0 \text{ or } c_1 \geq c_0 = 0\right\}\!
\]
whose Grothendieck group is a UFD, namely $\zz[x]$. Note that, even though $x \in \mathcal{A}(S)$ by Lemma~\ref{lemma: atoms}, $x$ is not a prime element of $S$. In fact, we have that $x \,|_{S}\, x^4$ but $x \nmid_S -x^2$. Consequently, the semidomain $S$ is not a UFS. So, in general, the unique factorization property does not descend from $\mathcal{G}(S)$ to $S$.

Next, we show that a semisubtractive semidomain $S$ is a UFS if and only if the integral domain $\mathcal{G}(S)$ is a UFD and $\mathcal{U}(S)\in\{\{0\},S\}$.

\begin{lemma}\label{lemma: primes ascend/descend}
Let $S$ be a semisubtractive semidomain, and let $p\in S$. The following statements hold.
\begin{enumerate}
    \item If $p$ is a prime element of $S$, then $p$ is a prime element of $\mathcal{G}(S)$.
    \item If $p$ is a prime element of $\mathcal{G}(S)$ and $\mathcal{U}(S)\in\{\{0\},S\}$, then $p$ is a prime element of $S$.
\end{enumerate}
\end{lemma}

\begin{proof}
    Suppose towards a contradiction that $p$ is a prime element of $S$ that is not prime in $\mathcal{G}(S)$. So, there exist $s,s' \in \mathcal{G}(S)$ such that $p \mid_{\mathcal{G}(S)} ss'$ but $p \nmid_{\mathcal{G}(S)} s$ and $p \nmid_{\mathcal{G}(S)} s'$. Let $s_1,s_2 \in S$ such that $s_1^2 = s^2$ and $s_2^2 = (s')^2$. Note that $p \nmid_S s_1$ and $p \nmid_S s_2$. Since $p$ is prime in $S$ and $p \mid_{\mathcal{G}(S)} s_1s_2$, there is no loss in assuming that there exists $s_3\in\mathscr{G}(S)$ such that $-s_3\in S$ and $ps_3=s_1s_2$. Thus,
	\[
	s_1^2s_2^2 = p^2(-s_3)^2.
	\]
	Since $p$ is a prime element of $S$, we have that either $p \mid_S s_1$ or $p \mid_S s_2$, a contradiction.

    As for statement (2), let $p$ be a prime element of $\mathcal{G}(S)$. If $\mathcal{U}(S)=S$, we have that $S=\mathcal{G}(S)$, so it follows that $p$ is a prime element of $S$. On the other hand, suppose $\mathcal{U}(S)=\{0\}$. It is easy to see that $p$ is a nonzero nonunit of $S$ (Lemma~\ref{lemma: units}). Now, let $s_1,s_2\in S$ such that $p\mid_S s_1s_2$. Then $p\mid_{\mathcal{G}(S)}s_1s_2$, so we may assume without loss of generality that $p\mid_{\mathcal{G}(S)}s_1$. In other words, there exists some $s'\in\mathcal{G}(S)$ such that $ps'=s_1$. If $-s'\in S$, then observe that $p(-s')=-s_1\in S$ which implies that $s_1\in\mathscr{U}(S)$, and hence $p\mid_S 0=s_1$. If $-s^{\prime}\not\in S$, then $s^{\prime}\in S$, and so it follows that $p\mid_S s_1$.
\end{proof}

\begin{theorem} \label{theorem: UFP}
	Let $S$ be a semisubtractive semidomain. Then $S$ is a UFS if and only if $\mathcal{G}(S)$ is a UFD and $\mathcal{U}(S)\in \{\{0\},S\}$.
\end{theorem}

\begin{proof}
Suppose $S$ is a UFS, and let $s$ be a nonzero nonunit element of $\mathcal{G}(S)$. Without loss of generality, assume that $s \in S$. Since $S$ is a UFS, we can write $s = p_1 \cdots p_n$, where $p_1, \ldots, p_n$ are atoms, and therefore prime elements, of $S$. By Lemma~\ref{lemma: primes ascend/descend}(1), we can write $s$ as a product of finitely many prime elements of $\mathcal{G}(S)$. Therefore, $\mathcal{G}(S)$ is a UFD. Now, suppose towards a contradiction that $\mathcal{U}(S)\notin\{\{0\},S\}$. Let $a$ be a nonzero nonunit element of $\mathcal{U}(S)$ whose unique factorization in $S$ is $p_1\cdots p_n$. Similarly, let $q_1\cdots q_m$ represent the unique factorization of $-a$ in $S$. If there is a bijective map $\varphi : [\![1,n]\!]\rightarrow [\![1,m]\!]$ such that $p_i=q_{\varphi(i)}\varepsilon_i$ for $\varepsilon_1,\ldots, \varepsilon_n\in S^{\times}$, then $\varepsilon_1\cdots\varepsilon_n=-1\in S$, a contradiction as $\mathcal{U}(S)\neq S$. Otherwise, if there is no such bijective map, $a^2$ has non-unique factorizations $p_1^2\cdots p_n^2$ and $q_1^2\cdots q_m^2$, also a contradiction. Consequently, $\mathcal{U}(S)\in\{\{0\},S\}$.

To tackle the reverse direction, observe that if $\mathcal{U}(S)=S$, then $S=\mathcal{G}(S)$, so $S$ is clearly a UFS. From now on, we may assume that $\mathcal{U}(S)=\{0\}$. Let $s$ be a nonzero nonunit element of $S$. Since $\mathcal{G}(S)$ is a UFD, we can write $s=p_1\cdots p_n$, where $p_1,\ldots,p_n$ are atoms, and therefore prime elements of $\mathcal{G}(S)$. Choose $e_1,\ldots,e_n\in\{0,1\}$ such that $p_i'=(-1)^{e_i}p_i$ is an element of $S$ for all $i\in[\![1,n]\!]$. If $e_1+\cdots+e_n$ is odd, then observe that $p_1'\cdots p_n'=(-1)^{e_1+\cdots+e_n}s=-s\in S$, which implies that $s\in\mathcal{U}(S)$. Therefore, we can assume that $e_1+\cdots+e_n$ is even. Since $-1$ is a unit of $\mathcal{G}(S)$, observe that $p_i'$ is a prime element of $\mathcal{G}(S)$ for all $i\in[\![1,n]\!]$. By Lemma~\ref{lemma: primes ascend/descend}(2), we can write $s$ as a product of finitely many primes, namely $p_1'\cdots p_n'$. Thus, $S$ is a UFS.
\end{proof}

For the rest of the section, we focus on the half-factorial property. Recall that a half-factorial semidomain $S$ is an atomic semidomain satisfying that, for every $s \in S^*$\!, all factorizations of $s$ share the same length (i.e., $|\mathsf{L}(s)| = 1$ for all $s \in S^*$). While UFSs are clearly HFSs, the reverse implication does not hold as the following example illustrates. 

\begin{example}
	Let $D = \zz[M]$, where $M = \langle (1,n) \mid n \in \nn \rangle \subseteq \nn_0^2$. Clearly, the (cancellative and commutative) monoid $M$ is torsion-free, which implies that $D$ is an integral domain by \cite[Theorem~8.1]{rG84}. Consider the semisubtractive semidomain $R = \{f \in D \mid f(0,0) \in \nn_0\}$, and let $S = \{f \in R \mid f(0,0) > 0\}$. Since $S$ is a multiplicative subset of $R$ (i.e., a submonoid of $(R^*,\cdot)$), we can consider the localization of $R$ at $S$, which we denote by $S^{-1}R$. We already established that $S^{-1}R$ is a semisubtractive semidomain (Lemma~\ref{lemma: localization of semisubtractive semidomains}). It is a known fact, as stated in \cite[Example~4.23]{CGG21}, that $M$ is an HFM that is not a UFM. As a result, the semidomain $S^{-1}R$ is not a UFS. In fact, if $a_1 + \cdots + a_n$ and $a_1' + \cdots + a_{n'}'$ are two distinct factorizations of $m \in M$, then we have that
	\[
	\frac{x^{a_1}}{1} \cdots \frac{x^{a_n}}{1} \hspace{.5 cm} \text{ and } \hspace{.5 cm} \frac{x^{a_1'}}{1} \cdots \frac{x^{a_{n'}'}}{1}
	\]
	are two distinct factorizations of the element $\frac{x^m}{1}\in S^{-1}R$. We now show that $S^{-1}R$ is an HFS. Since all the elements of $M \setminus \{(0,0)\}$ that are not atoms are divisible by $(1,1)$ and $x^{(1,1)}/1$ is irreducible in $S^{-1}R$, the semidomain $S^{-1}R$ is atomic. Now let $f/g$ be a nonzero nonunit element of $S^{-1}R$. Since $f/g \simeq_{S^{-1}R} f/1$, there is no loss in assuming that $g = 1$. Write $f = c_kx^{m_k} + \cdots + c_1x^{m_1}$, where $m_k > \cdots > m_1 > (0,0)$ in the lexicographic order. Let
	\[
	\frac{f_1}{g_1} \cdots \frac{f_n}{g_n}\hspace{.5 cm} \text{ and } \hspace{.5 cm}	\frac{f'_1}{g'_1} \cdots \frac{f'_{n'}}{g'_{n'}} 
	\]  
	be two distinct factorizations of $f/1$ in $S^{-1}R$, and suppose by way of contradiction that $n \neq n'$. Observe that if $f'/g'$ is an atom of $S^{-1}R$, then writing $f' = d_lx^{o_l} + \cdots + d_1x^{o_1}$ with $o_l > \cdots > o_1 > (0,0)$ in the lexicographic order, we have that $o_1 \in \mathcal{A}(M)$ because, again, all the elements of $M\setminus \{(0,0)\}$ that are not atoms are divisible by $(1,1)$. Consequently, the element $m_1 \in M$ has two factorizations of lengths $n$ and $n'$, which contradicts that $M$ is an HFM. Therefore, $S^{-1}R$ is a semisubtractive half-factorial semidomain that is not factorial.   
\end{example}

We are now in a position to describe the semisubtractive semidomains that are half-factorial.

\begin{theorem} \label{theorem: half-factoriality}
	Let $S$ be a semisubtractive semidomain. Then $S$ is an HFS if and only if $\mathcal{G}(S)$ is an HFD and $\mathcal{A}(S) = S \cap \mathcal{A}(\mathcal{G}(S))$. 
\end{theorem}

\begin{proof}
	Proving the reverse implication is straightforward, so we leave the details to the reader. Suppose now that $S$ is an HFS. We already established that $S \cap \mathcal{A}(\mathcal{G}(S)) \subseteq \mathcal{A}(S)$ (Lemma~\ref{lemma: atoms}). We now show that the reverse inclusion holds (given that $S$ is an HFS). Assume by way of contradiction that there exists $a \in \mathcal{A}(S) \setminus \mathcal{A}(\mathcal{G}(S))$. By Lemma~\ref{lemma: atoms}, we have that $-a \in S^* \setminus (S^{\times}\cup \mathcal{A}(S))$. Consequently, the element $a^2$ has two factorizations of different lengths in $S$, namely $a\cdot a$ and $(a_1 \cdots a_t)^2$ (for some $t \in \nn_{>1}$) with $a_1 \cdots a_t \in \mathsf{Z}_S(-a)$. This contradicts the fact that $S$ is half-factorial. Hence $\mathcal{A}(S) = S \cap \mathcal{A}(\mathcal{G}(S))$. 
	
	Observe that $\mathcal{G}(S)$ is atomic (in fact, a BFD) by Theorem~\ref{theorem: ACCP, BFP, and FFP}. Let $s$ be a nonzero nonunit element of $\mathcal{G}(S)$, and let 
	\[
		a_1 \cdots a_n b_1 \cdots b_m \hspace{.5 cm} \text{ and } \hspace{.5 cm} a_1' \cdots a_l'b_1' \cdots b_k'
	\]
	be two factorizations of $s$ in $\mathcal{G}(S)$ satisfying that $a_1, \ldots, a_n, a_1', \ldots, a_l' \in \mathcal{A}(S)$ and $b_1, \ldots, b_m, b_1', \ldots, b_k' \notin \mathcal{A}(S)$ for some $n, m, l, k \in \nn_0$ (for the rest of the proof, we proceed under the assumption that expressions such as $a_1 \cdots a_0$ represent the empty product, which is equal to $1$). Thus, we have that
	\[
	a_1^2 \cdots a_n^2(-b_1)^2 \cdots (-b_m)^2 \hspace{.3 cm}\text{and} \hspace{.3 cm} (a_1')^2 \cdots (a_l')^2 (-b_1')^2 \cdots (-b_k')^2
	\]
	are two factorizations in $S$ of the element $s^2$, which is not a unit of $S$ by Lemma~\ref{lemma: units}. Since $S$ is an HFS, we have that $2(m + n) = 2(l + k)$, which implies that $m + n = l + k$. We can then conclude that $\mathcal{G}(S)$ is an HFD. 
\end{proof}

In Theorem~\ref{theorem: half-factoriality}, the assumption that $\mathcal{A}(S) = S \cap \mathcal{A}(\mathcal{G}(S))$ is not superfluous. Consider the following example.

\begin{example}
	As in Example~\ref{ex: atom in S not in gp(S)}, consider the semidomain 
	\[
	S = \left\{c_nx^n + \cdots + c_1x + c_0 \in \zz[x] \mid \text{ either } c_0 > 0 \text{ or } c_1 \geq c_0 = 0\right\}\!.
	\]
	We already established that $S$ is a semisubtractive semidomain such that $\mathcal{G}(S) \cong \zz[x]$ and $-x^2 \in \mathcal{A}(S) \setminus \mathcal{A}(\zz[x])$. By virtue of Theorem~\ref{theorem: ACCP, BFP, and FFP}, we have that $S$ is atomic. However, note that $\{2,4\} \subseteq \mathsf{L}(x^4)$, which implies that $S$ is not half-factorial.  
\end{example}

We conclude this section providing a large class of proper half-factorial (i.e., half-factorials that are not factorials) semisubtractive semidomains.

An integral domain $(R,+,\cdot)$ is called an \emph{ordered integral domain} if there exists an order relation $\leq$ on $R$ satisfying the following two conditions:
\begin{enumerate}
	\item[(i)] $a \leq b$ implies $a + r \leq b + r$ for all $a,b,r \in R$;
	\item[(ii)] $a \leq b$ and $r \geq 0$ imply $ar \leq br$ for all $a,b,r \in R$.
\end{enumerate}
Given an ordered integral domain $R$, we denote by $R^+$ the subset of $R$ consisting of all elements that are greater than or equal to $0$.

\begin{example}
	Let $R$ be an ordered UFD. Consider the semisubtractive semidomain $S = R^+ + xR[x]$, and notice that $\mathcal{G}(S) \cong R[x]$. Let $f$ be an irreducible of $S$ satisfying that $-f \notin S$ (e.g., $x + 1$). Since $R[x]$ is a UFD, we can write $-f = a_1 \cdots a_k$ with $a_i \in \mathcal{A}(R[x])$ for every $i \in \llbracket 1,k \rrbracket$. As $-f \not \in S$, we have $f = (-1)^{n_1}a_1 \cdots (-1)^{n_k}a_k$, where each factor $(-1)^{n_i}a_i \in S$ for every $i \in \llbracket 1,k \rrbracket$; note that, by Lemma~\ref{lemma: units}, none of these factors is a unit of $S$. Consequently, we have that $k = 1$ which, in turn, implies that $f \in \mathcal{A}(R[x])$. Hence $\mathcal{A}(S) = S \cap \mathcal{A}(\mathcal{G}(S))$. By Theorem~\ref{theorem: half-factoriality}, it follows that $S$ is half-factorial. Note that $S$ is not factorial as $x \cdot x$ and $(-x)\cdot(-x)$ are two factorizations of the same element of $S$; however, we have that $x \not\simeq_S -x$.
\end{example}

As a final remark, we note that the results presented in this paper (Lemma~\ref{lemma: units}, Lemma~\ref{lemma: atoms}, Theorem~\ref{theorem: Furstenberg}, Lemma~\ref{lemma: atomicity}, Theorem~\ref{theorem: atomicity iff}, Theorem~\ref{theorem: ACCP, BFP, and FFP}, Lemma~\ref{lemma: primes ascend/descend}, Theorem~\ref{theorem: UFP}, and Theorem~\ref{theorem: half-factoriality}) can be generalized to submonoids N of a monoid M satisfying that, for some $\varepsilon \in M$, $\varepsilon^2 = 1$ and $M = \{n, n\varepsilon \mid n \in N\}$.

\section*{Acknowledgments}

The authors kindly thank their mentor, Dr. Harold Polo, for his dedicated guidance and support throughout the research period. The authors thank an anonymous referee for providing several useful suggestions, including the ideas for Lemma~\ref{lemma: atomicity} and the reverse implication of both Theorem~\ref{theorem: atomicity iff} and Theorem~\ref{theorem: UFP}, which have considerably improved the quality of this paper. The authors extend their gratitude to the MIT PRIMES program for making this opportunity possible.

\end{document}